\documentclass{gOPT2e}
\usepackage{latexsym}
\usepackage{graphics}
\usepackage{amssymb}
\newcommand{\R}{I\!\!R}
\newcommand{\N}{I\!\!N}

\begin{document}
\title{G-coupling functions}
\author{MORALES SILVA D.M.$^1$, RUBINOV A.M. $^2$ and SOSA W.$^3$\\\  \\
$^1$ SITMS, University of Ballarat, \\
Victoria 3353, Australia. E-mail: 2555675@fs2.ballarat.edu.au\\
$^2$ SITMS, University of Ballarat, \\
Victoria 3353, Australia. E-mail: a.rubinov@ballarat.edu.au\\
$^3$ IMCA, Universidad Nacional de Ingenier\'{i}a, \\
Jir\'{o}n Ancash 536, Lima 1, Lima, Per\'{u}, e-mail:
sosa@uni.edu.pe }

\date{March 2006}

\maketitle

\begin{abstract}
GAP functions are useful for solving Optimization Problems, but
the literature contains a variety of different concepts of GAP
functions. It is interesting to point out that these concepts have
many similarities. Here we introduce G-coupling functions, thus
presenting a way to take advantage of these common properties.

\vspace{0.3cm} \noindent{\it Keywords}: General conjugation
theory; non convex optimization; equilibrium problems; gap
functions.

\vspace{0.3cm} \noindent{\it 2000 Mathematics Subject
Classifications}:
\end{abstract}
\section{Introduction}
For solving non-convex optimization problems, a tool that is
becoming more important is {\bf{generalized conjugation}}. In this
paper we introduce a family of coupling functions, the G-coupling
functions, which will allow us to see in a different way duality
schemes. The usual properties found in the literature (\cite{RUB},
\cite{SOSA}) are related to a fixed coupling function, but here we
consider (for a specified function $f$) a family of coupling
functions.

These coupling functions are motivated by gap functions. It is
interesting to point out, that many of these (gap) functions have
similar properties. However, in some cases they are functions of
one vector and it is important, since they are linked to specified
optimization problems, that those functions have zeros.

G-coupling functions will be defined as functions in two variables
and they might not have zeros. Even more, given a specified proper
function $f$, it is shown that a sub-family of this family of
coupling functions satisfies many interesting properties.

In Section 2, we describe how many gap functions have similar
properties, which are useful for the definition of G-coupling
functions.

In Section 3, it is found the definition of G-coupling function
with properties related to generalized conjugation using this
family of functions and a fixed proper function $f$.

In Section 4, it can be seen how these ideas can be applied in the
Equilibrium Problem.

\section{Motivation}
In several works already published, there can be found definitions
of GAP functions for particular problems. Now we present 3
concrete examples.

\noindent In \cite{BUR}, it is consider the following Variational
Inequality Problem:
$$(VIP)\mbox{ Find }x_0\in C,\mbox{ such that, } \exists y^*\in T(x_0)
\mbox{ with }\langle y^*,x-x_0\rangle \geq 0\ \forall x\in C,$$
where $T$ is a maximal monotone correspondence (i.e., $\langle
u-v,x-y \rangle \geq 0$ for every $u\in T(x),\ v\in T(y)$ with
$x,y\ \in C$ and if there exists $v$, such that $\langle u-v,x-y
\rangle \geq 0$, for all $x,y\ \in C$ and for all $u\in T(x)$,
then $v\in T(y)$). It is found as a GAP function the following
one: $$h_{T,C}(x):=\sup_{(v,y)\in G_C(T)}\langle v,x-y\rangle,$$
where $G_C(T)=\{(v,y):\ v\in T(y),\ y\in C\}$ and $C$ is a
non-empty closed convex set. This function happens to be
non-negative and convex, and it is equal to zero only in solutions
of $(VIP)$.

\noindent The theory of the Equilibrium Problem begun with the
paper written by Blum and Oettli \cite{Blum-Oettli}:
$$(EP)\mbox{ Find }x\in K,\mbox{ such that }f(x,y)\geq 0,\ \forall
y\in K,$$ where $K\subset \R^n$ is a non-empty closed convex set
and $f:K\times K\rightarrow \R$ is a function that satisfies:
\begin{enumerate}
\item[i)] $f(x,x)=0$, for all $x\in K$.
\item[ii)] $f(x,\cdot):K\rightarrow \R$ is convex and l.s.c.
\item[iii)] $f(\cdot,y):K\rightarrow \R$ is u.s.c.
\end{enumerate}
The GAP function is defined as (\cite{MOS}, \cite{SOSA}):
$$g_f(y):=\left\{
\begin{array}{cc}
   \displaystyle\sup_{x\in K}f(x,y) & if\ y\in K  \\
   +\infty & \mbox{in other case}.
\end{array}
\right.$$ In this case, the function $g_f$ is non-negative, convex
and l.s.c. and if it vanishes at $x_0$, then $x_0$ is a solution
of $(EP)$.

\noindent In \cite{YANG}, the Extended Pre-Variational Inequality
Problem is considered:
$$(EPVIP)\mbox{ Find }x_0\in \R^n,\mbox{ such that }\langle
F(x_0),\eta(x,x_0)\rangle \geq f(x_0)-f(x),\ \forall x\in \R^n,$$
where $F:\R^n \rightarrow \R^n$, $\eta: \R^n\times \R^n
\rightarrow \R^n$ and $f:\R^n \rightarrow \R\cup\{+\infty\}$. In
this work, the GAP function is
$$\min_{y\in \R^n}[\langle F(x),\eta(y,x)\rangle -f(x)+f(y)],$$ which is
non-positive and it only reaches the value zero in solutions of
$(EPVIP)$.

In all these examples, gap functions are used to transform an
special Equilibrium Problem (for example the VIP is a particular
class of EP) into a minimization problem.

Now our attention is focused in using coupling functions that
could be related, at least in some general aspect, to GAP
functions. Therefore these functions must link both primal and
dual variables. Since these coupling functions must be related to
a sense of ``gap", we consider these functions as non-negative and
with 2 arguments.

Let us remember that for the minimization problem, the convex
conjugation theory allows us to generate a dual problem and there
is implicit another concept of gap function (see
\cite{J.P.SOSA.OC}): consider
$$\alpha =\inf[f(x):x\in \R^n]. \qquad (P)$$ Define a function
$\varphi:\R^n\times \R^p\rightarrow \overline{\R}$, where
$\overline{\R}=\R\cup\{-\infty,+\infty\}$, satisfying
$$\varphi(x,0)=f(x),\ \forall x\in \R^n.$$ Then $\varphi$ will
be called a perturbation function and the function
$h:\R^p\rightarrow \overline{\R}$ defined by
$$h(u)=\inf_{x\in \R^n}\varphi(x,u)$$ will be called the marginal
function. Observe that $$\alpha=h(0)=\inf_{x\in
\R^n}\varphi(x,0)=\inf_{x\in \R^n}f(x).$$ Considering now
$h^{**}$, the convex bi-conjugate (see \cite{ROCK}) of $h$ one
has:
$$h^{**}(0)\leq h(0)=\alpha$$ where $$h^{**}(0)=\sup[\langle
u^*,0\rangle -h^*(u^*): u^*\in \R^p].$$ Then, making
$-\beta=h^{**}(0)$, one has
$$\beta=\inf_{u^*\in \R^p}h^*(u^*).\qquad (Q)$$ $(Q)$ is called dual problem of
$(P)$ and in general we have $-\beta\leq\alpha$. It is said that
there is no duality gap whenever $h^{**}(0)=h(0)$. It is easy to
prove that $h^*(u^*)=\varphi^*(0,u^*)$, and if we define the
function $k:\R^n\rightarrow \overline{\R}$ by $\displaystyle
k(x^*):=\inf_{u^*\in \R^p}\varphi^*(x^*,u^*)$, then $\beta=k(0)$.

\noindent This analysis is summarized in the following scheme:
$$\begin{array}{rclcrcl}
  \alpha & = & \inf f(x) \qquad (P) & & \beta & = & \inf h^*(u^*) \qquad (Q) \\
  \varphi(x,0) & = & f(x),\ \forall x\in \R^n & & \varphi^*(0,u^*) & = & h^*(u^*),\
  \forall u^*\in \R^p \\
  h(u) & = &\displaystyle \inf_x \varphi(x,u) & & k(x^*) & = &\displaystyle
  \inf_{u^*}\varphi^*(x^*,u^*) \\
  \alpha & = & h(0) & & \beta & = & k(0)
\end{array} $$
$$-\beta \leq \alpha.$$
If $h$ is proper and convex, a necessary and sufficient condition
for ensuring that there will be no duality gap ($-\beta =\alpha$)
is that $h$ be l.s.c. at 0 (in general $\varphi$ l.s.c. does not
imply that $h$ would be l.s.c.).

Further more, if $h$ is convex, l.s.c. and $0\in ri(dom(h))$, then
$\alpha=-\beta$ and the dual problem has at least one optimal
solution, and if $\overline{u^*}$ is an optimal solution of $(Q)$
and $\varphi=\varphi^{**}$, then
$$\overline{x} \mbox{ is an optimal solution of }(P)
\Longleftrightarrow f(\overline{x})+h^*(\overline{u^*})=0.$$
Consider now the function $g:\R^n \times \R^p \rightarrow
\overline{\R}$ defined by:
$$g(x,u^*)=f(x)+h^*(u^*).$$ This function vanishes at
$(x_0,u^*_0)$ if and only if $x_0$ solves the primal problem and
$u^*_0$ solves the dual one. In addition, this function is
non-negative and if the first variable is kept fixed, the function
is convex and l.s.c. It is clear now, which properties are
satisfied for many gap functions.

\section{Generalized Conjugation}
\begin{definition}
A non-negative function $g:\R^n \times \R^m \rightarrow
\overline{\R}$ will be called a G-coupling function, if there
exists a non-empty closed set $C\subset \R^m$ such that:
\begin{enumerate}
\item[(D1)] $dom(g)=\R^n\times C$, which means that $g:\R^n
\times C \rightarrow \R$ and $g(x,x^*)=+\infty$ for every
$x^*\notin C$.
\item[(D2)] $\displaystyle \inf_{x\in \R^n,\ x^*\in C}g(x,x^*)=0$.
\end{enumerate}
\end{definition}

Define $\mathcal{F}^{n,m}=\{g:\R^n\times \R^m \rightarrow
\overline{\R} : \ g \mbox{ is a G-coupling function}\}$.

Not every G-coupling function has zeros:\vspace{2ex}

\noindent {\bf Example:} Define on $\R^2$
$$g(x,x^*)=\exp(x+x^*).$$ \verb|   | Then $g\in \mathcal{F}^{1,1}$
($C=\R$) is continuous and it does not have any zeros.

\begin{proposition}\label{pr1}
Let $g\in \mathcal{F}^{n,n}$ and $C\subset \R^n$ be the non-empty
closed set linked to $g$. The following statements hold
\begin{enumerate}
\item[i)] If $\displaystyle \lim_{\|(x,x^*)\|\rightarrow
+\infty}g(x,x^*)=+\infty$, with $(x,x^*)\in \R^n\times C$ and $g$
is l.s.c. then
$$\bigcap_{\varepsilon >0}\frac{}{}(S_g(\varepsilon)=\{(x,x^*)\in \R^n\times C: g(x,x^*)\leq
\varepsilon\}) \neq \emptyset.$$
\item[ii)] If $g$ is lsc on $\R^n\times C$ and there
exists $M>0$ such that $S_g(M)$ is bounded, then
$$\{(x,x^*):g(x,x^*)=0\}=\bigcap_{\varepsilon
>0}S_g(\varepsilon) \neq \emptyset.$$
\end{enumerate}
\end{proposition}
\begin{proof}
\begin{enumerate}
\item[i)] It is well known (see \cite{AUS-TEB}) that
$\displaystyle \lim_{\|(x,x^*)\|\rightarrow
+\infty}g(x,x^*)=+\infty$ is equivalent to the fact that all the
level sets $S_g(\varepsilon)$ are bounded. The statement follows
from the lsc of $g$.
\item[ii)] It follows from the lsc of $g$.
\end{enumerate}
\end{proof}

Let us turn our attention now to how the family of functions
$\mathcal{F}^{n,m}$ will allow us to establish duality schemes
for, at least, the minimization problem. It is important to point
out that in the following we consider an unusual type of duality:
$f$ is kept fixed and $g\in \mathcal{F}^{n,m}$, for a given $m\in
\N$, is variable.

Consider a proper function $f:\R^n \rightarrow \R \cup
\{+\infty\}$. For a given $m\in \N$ take $g\in \mathcal{F}^{n,m}$.
Define $f^g:C\rightarrow \R \cup \{+\infty\}$ and
$f^{gg}:\R^n\rightarrow \R\cup\{+\infty\}$ as follows (for example
see \cite{OS}, \cite{RUB} and references therein):
$$f^g(x^*):=\sup_{x\in \R^n}\{g(x,x^*)-f(x)\}\ \forall x^*\in
C,$$ $$f^{gg}(x):=\sup_{x^*\in C}\{g(x,x^*)-f^g(x^*)\}\ \forall
x\in \R^n,$$ where $C$ is the closed set linked to $g$. In some
cases, it would be better to consider a $g\in \mathcal{F}^{n,m}$
which satisfies:
\begin{enumerate}
\item[(D3)] $C$ is convex and $g(x,\cdot):C \rightarrow \R$ is a
convex and l.s.c. function for each $x$ in $\R^n$.
\end{enumerate}
With this, we have the following:

\begin{lemma} Let $f:\R^n \rightarrow \R \cup
\{+\infty\}$ be a proper function and given $m\in \N$. If $g\in
\mathcal{F}^{n,m}$, then
$$f^g(x^*)+f(x)\geq g(x,x^*)\geq 0,\ \forall (x,x^*)\in \R^n\times C,$$ which implies
$$f(x)\geq -f^g(x^*),\ \forall (x,x^*)\in \R^n\times C.$$
Furthermore, if $g$ satisfies $(D3)$, then $f^g$ is a convex l.s.c
function.
\end{lemma}

Unless it is mentioned, not every $g\in \mathcal{F}^{n,m}$
satisfies $(D3)$.

It would be interesting to know which condition either a
G-coupling function $g$ or the function $f$ must satisfy in order
that the function $f^g$ be proper, because with this one would
have a non-trivial function related to $f$. The following lemma
ensures the existence of such a function $g\in \mathcal{F}^{n,m}$
for any $m\in \N$, taking as a starting point a natural condition
on $f$ which must be imposed if $f$ is the objective function of a
minimization problem.

\begin{lemma} Let $f$ be as before. Then $f$ is bounded from below
if and only if, for every $m\in \N$, there exists $g\in
\mathcal{F}^{n,m}$ such that $f^g$ is proper.
\end{lemma}
\begin{proof}
\begin{enumerate}\item[$\bullet$] Suppose that $\inf
f>-\infty$, then for a $m_0\in \N$ fixed, consider $g\in
\mathcal{F}^{n,m_0}$ as follows:
$$\displaystyle g(x,x^*)=\left\{\begin{array}{cl}
  \|x^*\| & \mbox{if }x\in dom(f) \\
  0 & \mbox{if }x\notin dom(f),
\end{array}\right.$$ ($C=\R^{m_0}$) thus
$$f^{g}(x^*)=\|x^*\|-\inf f\ \forall x^*\in \R^{m_0},$$
which is clearly a proper function and since $m_0\in \N$ was fixed
arbitrarily, the result is satisfied for every $m\in \N$.
\item[$\bullet$] Take $m_0\in \N$ and $g\in
\mathcal{F}^{n,m_0}$ such that $f^g$ is proper. Let us suppose
that $\inf f=-\infty$, from \cite{SOSA} we can see that this
implies that $\inf f^{gg}=-\infty$. Then:
$$-\infty=\inf f^{gg}=\inf_{x\in \R^n} \left(\sup_{x^*\in
C}[g(x,x^*)-f^g(x^*)]\right)\geq$$ $$\sup_{x^*\in C }\left(
\inf_{x\in \R^n} [g(x,x^*)-f^g(x^*)] \right)\geq \sup_{x^*\in C
}(-f^g(x^*)) =-\inf_{x^*\in C}f^g(x^*),$$ which means
$\displaystyle -\infty\geq -\inf_{x^*\in C}(f^g(x^*))$. Then
$\displaystyle \inf_{x^*\in C}f^g(x^*)=+\infty$, which implies
that $f^g$ is not proper and we have a contradiction. Therefore,
$\inf f>-\infty$.
\end{enumerate}
\end{proof}

Notice that this proof also states, in particular, that there
exists $g\in \mathcal{F}^{n,m}$ for every $m\in \N$ which
satisfies $(D3)$ and $f^g$ is proper.

Henceforth, consider only functions $f$ such that $\inf f>-\infty$
and for some fixed $m\in \N$, $g\in \mathcal{F}^{n,m}$ will be
such that $f^g$ is proper.

Let it be $\mathcal{F}^{n}=\{f:\R^n\rightarrow \R\cup\{+\infty\},\
f\mbox{ is proper, } \inf f>-\infty\}$ and\\
$\gamma_{g,f}:\R^n\times \R^m\rightarrow \R\cup\{+\infty\}$
defined by:
$$\gamma_{g,f}(x,x^*)=\left\{\begin{array}{cl}
  f(x)+f^g(x^*) & \mbox{if }x^*\in C \\
  +\infty & \mbox{otherwise.}
\end{array}\right.$$ with $m\in \N$, $g\in \mathcal{F}^{n,m}$
and $f\in \mathcal{F}^n$. Define
$$\mathcal{F}_f^{n,m}:=\{g\in \mathcal{F}^{n,m}/ f^g \mbox{ is
proper and }\inf \gamma_{g,f}=0\},$$ observe that $\gamma_{g,f}$
might not be in $\mathcal{F}^{n,m}$, since $\gamma_{g,f}$ can take
the value $+\infty$ for some $(x,x^*)\in \R^n\times C$.

\begin{lemma}
$\mathcal{F}_f^{n,m}$ is non-empty for all $m\in \N$.
\end{lemma}
\begin{proof}
Given $m\in \N$, define $g\in \mathcal{F}^{n,m}$ by:
$$\displaystyle g(x,x^*)=\left\{\begin{array}{cl}
  \|x^*\| & \mbox{if }x\in dom(f) \\
  0 & \mbox{if }x\notin dom(f).
\end{array}\right.$$ It is easy to check that $g$ belongs to
$\mathcal{F}_f^{n,m}$ with $C=\R^m$ (this example also proves that
functions can be found in $\mathcal{F}_f^{n,m}$ which satisfy
$(D3)$).
\end{proof}

Now consider $$(P)\ \min_x f(x)$$ with $f\in \mathcal{F}^n$.
Taking $g\in \mathcal{F}_f^{n,m}$, define the dual problem related
to $g$:
$$(D_g)\ \min_{x^*\in C }f^g(x^*).$$ Since
$$\inf_{(x,x^*)\in \R^n\times C}\gamma_{g,f}(x,x^*)=\inf_{x\in
\R^n} f(x)+\inf_{x^*\in C}f^g(x^*)=0$$
$$\Longrightarrow \inf_{x\in \R^n} f(x)=-\inf_{x^*\in
C}f^g(x^*)=\sup_{x^*\in C}[-f^g(x^*)].$$ This means that there is
no duality gap between the primal problem $(P)$ and its dual
$(D_g)$ for every $g\in \mathcal{F}_f^{n,m}$.\\
\\ The next theorem states that given $n,m\in \N$, the
correspondence defined by
$$\begin{array}{cccc}
  \mathbf{F}: & \mathcal{F}^n & \rightrightarrows &
  \mathcal{F}^{n,m}\\
   & f & \mapsto & \mathbf{F}(f)=\mathcal{F}_f^{n,m},
\end{array}$$ is a closed correspondence (see \cite{ZANG}).

\begin{theorem}
Take $f\in \mathcal{F}^n$ and $C\subset \R^m$ the non-empty closed
convex set. If there exist $f_k:dom(f)\rightarrow \R$, $g_k:\R^n
\times C\rightarrow \R$, sequences of functions ($k\in \N$), such
that $f_k$ converges uniformly to $f$ (in $dom(f)$), $g_k\in
\mathcal{F}_{f_k}^{n,m}$ satisfies $(D3)$ for every $k\in \N$ and
$g_k$ converges uniformly to a function $g$ (in $\R^n\times C$),
then $g\in \mathcal{F}_f^{n,m}$ and it satisfies $(D3)$ (extend
$g$ to $\R^n\times \R^m$ by only taking $g(x,x^*)=+\infty$
whenever $x^*\notin C$).
\end{theorem}
\begin{proof}
Let us prove first that $g\in \mathcal{F}^{n,m}$. Since $g_k$
converges uniformly to $g$, given $\varepsilon>0$, there exists
$N\in \N$ such that if $k\geq N$ then
$$|g_k(x,x^*)-g(x,x^*)|< \varepsilon,\qquad \forall (x,x^*)\in
\R^n\times C.$$
$$\mbox{Hence }
g_k(x,x^*)-\varepsilon<g(x,x^*)<g_k(x,x^*)+\varepsilon,\ \forall
(x,x^*)\in \R^n\times C.$$ Taking $\displaystyle \inf_{x,x^*\in
C}$ (remember that $\inf g_k =0$ for all $k\in \N$):
$$-\varepsilon<\inf_{x,x^*\in C}g(x,x^*)< \varepsilon.$$ Then $|\inf
g|<\varepsilon$. And since $\varepsilon>0$ is arbitrary, one has
that $\inf g=0$.

Now we prove that $g(x,\cdot):C\rightarrow \R$ is convex and
l.s.c. for all $x\in \R^n$.\begin{enumerate}
\item[$\bullet$] $g(x,\cdot)$ is convex: let $x\in \R^n$ be
fixed arbitrarily. Since for all $k\in \N$, $g_k(x,\cdot)$ is
convex, one has that given $x_1^*,x_2^*\in C$ and $t\in [0,1]$:
$$g_k(x,tx_1^* + (1-t)x_2^*)\leq
tg_k(x,x_1^*)+(1-t)g_k(x,x_2^*).$$ Making $k\rightarrow +\infty$:
$$g(x,tx_1^* + (1-t)x_2^*)\leq tg(x,x_1^*)+(1-t)g(x,x_2^*),$$
which proves that $g(x,\cdot)$ is convex.\item[$\bullet$]
$g(x,\cdot)$ is l.s.c.: take $\lambda <g(x,x^*)$. There exists
$N\in \N$ such that
$$|g_N(y,y^*)-g(y,y^*)|<\varepsilon,\ \forall (y,y^*)\in
\R^n\times C,$$ where $\displaystyle \varepsilon=
\frac{g(x,x^*)-\lambda}{2}$.
$$\mbox{Hence } \lambda< \lambda + \varepsilon=g(x,x^*)-\varepsilon<g_N(x,x^*).$$
Since $g_N(x,\cdot)$ is l.s.c., then there exists $V(x^*)\subset
C$, a neighborhood of $x^*$, such that if $y^*\in V(x^*)$ then
$$\lambda+\varepsilon<g_N(x,y^*).$$ Reducing $g(x,y^*)$:
$$\lambda+\varepsilon-g(x,y^*)<g_N(x,y^*)- g(x,y^*)<\varepsilon.$$
Therefore, if $y^*\in V(x^*)$, then $\lambda<g(x,y^*)$. Thus
$g(x,\cdot)$ is l.s.c. for all $x\in \R^n$.\end{enumerate} This
proves that $g\in \mathcal{F}^{n,m}$.

It remains to prove that $g\in \mathcal{F}_f^{n,m}$. For doing
this, let us show that $(f_k^{g_k})_{k\in \N}$ converges uniformly
to $f^g$.

Let $\varepsilon>0$ and $N\in \N$ such that if $k\geq N$ then
$$|g_k(x,x^*)-g(x,x^*)|<\frac{\varepsilon}{4},\ \forall (x,x^*)\in \R^n\times C \mbox{
and }$$
$$|f_k(x)-f(x)|<\frac{\varepsilon}{4},\ \forall x\in \R^n.$$
Fix $k\geq N$ and take $x^*\in C$ arbitrarily, then
$$f_k^{g_k}(x^*)-\frac{\varepsilon}{2} < g_k(x',x^*)-f_k(x'), \mbox{
for some }x'\in \R^n.$$
$$\mbox{Hence } f_k^{g_k}(x^*)-\varepsilon<g_k(x',x^*)-f_k(x')
-\frac{\varepsilon}{2}<g(x',x^*)-f(x')\leq f^g(x^*),$$
$$\mbox{and so } f_k^{g_k}(x^*)-\varepsilon<f^g(x^*).$$
This proves that $f_k^{g_k}(x^*)-f^g(x^*)<\varepsilon.$ On the
other hand:
$$f^g(x^*)-\frac{\varepsilon}{2}<g(x'',x^*)-f(x''),\ \mbox{for some }x''\in \R^n,$$
$$\mbox{whence } f^g(x^*)-\varepsilon
<g(x'',x^*)-f(x')-\frac{\varepsilon}{2}<g_k(x'',x^*)-f_k(x''),$$
$$\mbox{and so } f^g(x^*)-\varepsilon< f_k^{g_k}(x^*).$$ With this,
it is shown that $-\varepsilon<f_k^{g_k}(x^*)-f^g(x^*)$ for an
arbitrary $x^*\in C$. Therefore $(f_k^{g_k})_{k\in \N}$ converges
uniformly to $f^g$ (in $C$), and it is immediate to see that $f^g$
is proper and
$$0\leq f(x)+f^g(x^*)\leq f_k(x)+f_k^{g_k}(x^*)+\varepsilon,\ \forall
x\in dom(f),\ x^*\in C$$ where $\varepsilon>0$ is arbitrary and
$k$ is large enough. Taking $\displaystyle\inf_{(x,x^*)\in
\R^n\times C}$ one has:
$$0\leq \inf_{(x,x^*)\in
\R^n\times C}(f(x)+f^g(x^*))\leq \varepsilon.$$ Therefore
$\displaystyle\inf_{(x,x^*)\in \R^n\times C}(f(x)+f^g(x^*))=0$ and
$g\in \mathcal{F}_f^{n,m}.$
\end{proof}

This theorem proves a more difficult situation, the case when
$g_k\in \mathcal{F}_{f_k}^{n,m}$ satisfy $(D3)$ for all $k\in \N$.
For the general case, just omit the two $\bullet$ items and change
$C$ for the non-empty closed set linked to $g$.

Consider now $g\in \mathcal{F}_f^{n,m}$. Define $l:\R^m
\rightarrow \R \cup \{+\infty\}$ as follows:
\begin{equation}\label{equa-l}l(x^*)=\left\{\begin{array}{cl}
  f^g(x^*) & \mbox{if }x^*\in C \\
  +\infty & \mbox{if }x^*\notin C.
\end{array}\right.\end{equation}

\begin{lemma}
Let $g\in \mathcal{F}_f^{n,m}$ satisfy $(D3)$, $l$ be defined like
in (\ref{equa-l}) and $l^*(x)=\displaystyle \sup_{x^*\in
\R^m}\{\langle x^*,x \rangle - l(x^*)\}$, if $0\in ri(dom (l^*))$.
Then the following statements are equivalent:
\begin{enumerate}
\item[i)]$\partial l^*(0)\neq \emptyset.$
\item[ii)] $(D_g)$ has a solution.
\end{enumerate}
\end{lemma}
\begin{proof}
$$\partial l^*(0)\neq \emptyset \Longleftrightarrow
\exists \overline{x}^* \mbox{ such that }
l^*(0)+l^{**}(\overline{x}^*)=l^*(0)+l(\overline{x}^*)=0.$$ This
is equivalent to
$$l(\overline{x}^*)=-l^*(0)=-\sup_{x^*\in \R^m}\{-l(x^*)\}=
\inf_{x^*\in \R^m}l(x^*)=\inf_{x^*\in C}f^g(x^*),$$ but
$\displaystyle \inf_{x^*\in C}f^g(x^*)=-\inf f\neq +\infty$ then
$l(\overline{x}^*)\neq +\infty$, therefore $\overline{x}^*\in C$
and $$l(\overline{x}^*)=f^g(\overline{x}^*)=\inf_{x^*\in
C}f^g(x^*).$$
\end{proof}

\begin{theorem}
$\overline{x}^*$ is a solution of $(D_g)$ and $\overline{x}$ is a
solution of $(P)$ if and only if
$\gamma_{g,f}(\overline{x},\overline{x}^*)=0$.
\end{theorem}
\begin{proof} $\overline{x}$ and $\overline{x}^*$ are solutions of
$(P)$ and $(D_g)$ respectively if and only if
$$f(\overline{x})= \inf f=-\inf f^g=-f^g(\overline{x}^*)
\Longleftrightarrow
f(\overline{x})+f^g(\overline{x}^*)=\gamma_{g,f}
(\overline{x},\overline{x}^*)=0.$$
\end{proof}

Define $m\left(\gamma_{g,f}\right):=\left\{(x,x^*)\in \R^n\times
C: \gamma_{g,f}(x,x^*)=0\right\}$, for the previous theorem,
$(x_0,x^*_0)$ belongs to $m\left(\gamma_{g,f}\right)$ if and only
if $x_0$ is a solution of $(P)$ and $x_0^*$ is a solution of
$(D_g)$. Take $f\in \mathcal{F}^n$ and $g\in \mathcal{F}^{n,m}_f$.
Define the set $R\left(\gamma_{g,f}\right)$, as follows:
$$R\left(\gamma_{g,f}\right):=\bigcap_{(x,x^*)\in \R^n\times C}
\left(S_{\gamma_{g,f}(x,x^*)}\left(\gamma_{g,f}\right)\right)^{\infty},$$
where $S_{\lambda}\left(\gamma_{g,f}\right)$ stands for the
$\lambda$-level set of the function $\gamma_{g,f}$
$$S_{\lambda}\left(\gamma_{g,f}\right):=\{(x,x^*):\gamma_{g,f}(x,x^*)
\leq\lambda\}.$$ and for $A\subset \R^p$,
$$A^{\infty}:=\{d\in \R^p:\ \exists\{x_k\}_{k\in
\N}\subset A,\{t_k\}_{k\in \N}\downarrow 0:\ t_kx_k\rightarrow
d\},$$ with the convention, $\emptyset^{\infty}=\{0\}$ the
following properties are satisfied (see \cite{AUS-TEB}):

\begin{lemma}\label{lemmAUSTEB} Let it be $A,B\subset \R^p$.
The following are true:
\begin{enumerate}\item[i)] $A^{\infty}$ is a closed cone.
\item[ii)] $A^{\infty}=\{0\}$ if and only if $A$ is bounded.
\item[iii)] If $A\subset B$ then $A^{\infty}\subset
B^{\infty}$.
\item[iv)] $A^{\infty}=(\overline{A})^{\infty}$.
\item[v)] If $A$ is a cone, then $A^{\infty}=\overline{A}$.
\item[vi)] If $(A_i)_{i\in I}$ is a family of sets in $\R^p$,
we have that $$\left(\bigcap_{i\in I}A_i\right)^{\infty}\subset
\bigcap_{i\in I}A_i^{\infty}.$$
\end{enumerate}
\end{lemma}

The proof of the next Lemma follows from: item (vi) of Lemma
\ref{lemmAUSTEB}, the definition of $R\left(\gamma_{g,f}\right)$
and the fact that $m\left(\gamma_{g,f}\right) =\displaystyle
\bigcap_{(x,x^*)\in \R^n\times C}
\left(S_{\gamma_{g,f}(x,x^*)}\left(\gamma_{g,f}\right)\right)$.

\begin{lemma}\label{l1}
$(m\left(\gamma_{g,f}\right))^\infty \subset
R\left(\gamma_{g,f}\right)$.
\end{lemma}

Note that, if $R\left(\gamma_{g,f}\right) = \{0\}$, then from
Lemma \ref{l1} $m\left(\gamma_{g,f}\right)$ is bounded (may be
empty).

\begin{lemma}\label{lc}
Let $\{\lambda_k\}_{k\in \N}$ be a sequence such that
$\displaystyle \lim_{k\to +\infty} \lambda_k = 0$ and $\lambda_k >
0$ $\forall k \in \N$. If $m\left(\gamma_{g,f}\right) =
\emptyset$, then $R\left(\gamma_{g,f}\right) =\displaystyle
\bigcap_{k\in \N} \left(S_{\lambda_k}\left(\gamma_{g,f}\right)
\right)^\infty$.
\end{lemma}
\begin{proof}
We only need to show that $\displaystyle \bigcap_{k\in \N}
\left(S_{\lambda_k}\left(\gamma_{g,f}\right) \right)^\infty
\subset R\left(\gamma_{g,f}\right)$. Indeed, take $u \in
\displaystyle \bigcap_{k\in \N}
\left(S_{\lambda_k}\left(\gamma_{g,f}\right) \right)^\infty$, then
$u \in \left(S_{\lambda_k}\left(\gamma_{g,f}\right)
\right)^\infty$ $\forall k \in \N$. For each $(x,x^*) \in
\R^n\times C$ (arbitrarily fixed), we have that $0 <
\gamma_{g,f}(x,x^*)$ (because $m\left(\gamma_{g,f}\right) =
\emptyset$). Since $\displaystyle \lim_{n\to +\infty} \lambda_n =
0$, we have that there exists $q\in \N$ such that $\lambda_q \leq
\gamma_{g,f}(x,x^*)$. So, $S_{\lambda_q}\left(\gamma_{g,f}\right)
\subset S_{\gamma_{g,f}(x,x^*)}\left(\gamma_{g,f}\right)$. It
implies that $\left(S_{\lambda_k}\left(\gamma_{g,f}\right)
\right)^\infty \subset
\left(S_{\gamma_{g,f}(x,x^*)}\left(\gamma_{g,f}\right)\right)^\infty$
and so $u \in
\left(S_{\gamma_{g,f}(x,x^*)}\left(\gamma_{g,f}\right)\right)^\infty$.
The statement follows.
\end{proof}

For the next Lemma, consider that the function $\gamma_{g,f}$ is
l.s.c. (it can be obtained, for example, if $f$ and $g(x,\cdot)$
$\forall x \in \R^n$ are l.s.c.).

\begin{lemma}\label{lpt}
$m\left(\gamma_{g,f}\right) = \emptyset$ if and only if
$(m\left(\gamma_{g,f}\right))^\infty \neq
R\left(\gamma_{g,f}\right)$
\end{lemma}
\begin{proof}
\begin{enumerate}
\item [$\rightarrow$] If $m\left(\gamma_{g,f}\right) = \emptyset$, then
$(m\left(\gamma_{g,f}\right))^\infty = \{0\}$ and
$\gamma_{g,f}(x,x^*) > 0$ $\forall (x,x^*) \in \R^n\times C$.
From, item (ii) of Proposition \ref{pr1}, we have that
$\left(S_{\gamma_{g,f}(x,x^*)}\left(\gamma_{g,f}\right)\right)$
are unbounded $\forall (x,x^*) \in \R^n\times C$. Now, consider
$\{\lambda_k\}_{k\in \N}$ such that $\displaystyle \lim_{k\to
+\infty} \lambda_k = 0$ and $\lambda_k > \lambda_{k+1}$ $\forall k
\in \N$. Take $u^k \in
\left(S_{\lambda_k}\left(\gamma_{g,f}\right) \right)^\infty$ with
$||u^k|| = 1$ $\forall k \in \N$. Since $\lambda_q \leq \lambda_k$
$\forall k \in \N$ and $\forall q \geq k$, we have that
$\left(S_{\lambda_q}\left(\gamma_{g,f}\right) \right)^\infty
\subset \left(S_{\lambda_k}\left(\gamma_{g,f}\right)
\right)^\infty$ and $\{u^q\}_{q\geq k} \subset
\left(S_{\lambda_k}\left(\gamma_{g,f}\right) \right)^\infty$
$\forall k \in \N$ and so any cluster point of $\{u^k\}_{k\in \N}$
belong to $\left(S_{\lambda_k}\left(\gamma_{g,f}\right)
\right)^\infty$ $\forall k\in \N$. From Lemma \ref{lc} we have
that any cluster point of $\{u^k\}_{k\in \N}$ belong to
$R\left(\gamma_{g,f}\right)$. So, the statement follows.
\item [$\leftarrow$] It is equivalent to show that: If $m\left(\gamma_{g,f}\right) \neq \emptyset$, then $R\left(\gamma_{g,f}\right) = (m\left(\gamma_{g,f}\right))^\infty$. Indeed, we know that $S_{0}\left(\gamma_{g,f}\right) = m\left(\gamma_{g,f}\right) \neq \emptyset$. So, we
have that, $R\left(\gamma_{g,f}\right) \subset
(S_{0}\left(\gamma_{g,f}\right))^\infty =
(m\left(\gamma_{g,f}\right))^\infty$. The statement follows.
\end{enumerate}
\end{proof}

\begin{theorem}
Given $n,m\in \N$, $f\in \mathcal{F}^n$ l.s.c. and
$g\in\mathcal{F}^{n,m}_f$ such that $g(x,\cdot)$ is l.s.c. for all
$x\in \R^n$. The following statements are equivalents.
\begin{enumerate}
\item $R\left(\gamma_{g,f}\right)=\{0\}$.
\item $m\left(\gamma_{g,f}\right)$ is nonempty and compact.
\end{enumerate}
\end{theorem}
\begin{proof}
We need to proof only the implication {\bf (i) implies (ii)}.
Indeed, from l.s.c of $f$ and $g(x,\cdot)$ $\forall x \in \R^n$,
we have that $\gamma_{g,f}$ is l.s.c. on $\R^n\times C$. Since
$R\left(\gamma_{g,f}\right)=\{0\}$, then from Lemma \ref{l1}
$\left(m\left(\gamma_{g,f}\right)\right)^\infty =\{0\}$. So, from
Lemma \ref{lpt} we have $m\left(\gamma_{g,f}\right) \neq
\emptyset$, here $m\left(\gamma_{g,f}\right)$ is closed. The
statement follows applying Lemma \ref{lemmAUSTEB} item (ii).
\end{proof}

At this point a natural question arises, if $g\in
\mathcal{F}_f^{n,m}$, would be there any kind of relation between
the optimal points and the optimal values of $f$ and $f^{gg}$? The
next lemma answers this.

\begin{lemma}
For a fixed $m\in \N$ and every $g\in \mathcal{F}^{n,m}_f$, the
following are satisfied:
\begin{enumerate}
\item[i)] $\inf f=\inf f^{gg},$
\item[ii)] if $x_0$ is a global minimum of $f$, then $x_0$ is a
global minimum of $f^{gg}$.
\end{enumerate}
\end{lemma}
\begin{proof} Remember that $f^{gg}$ is defined by:
$$f^{gg}(x)= \sup_{x^*\in C}\{g(x,x^*)-f^g(x^*)\},$$ where $C$ is
the closed set linked to $g$.
\begin{enumerate}
\item[i)] $\inf f^{gg}\leq \inf f$ is always true. On the other
hand $$f^g(x^*)+f^{gg}(x)\geq g(x,x^*)\geq 0,\ \forall x\in \R^n,\
x^*\in C,$$ which implies that $$\inf f^{gg}\geq -\inf_{x^*\in
C}f^g(x^*).$$ But, since $g\in \mathcal{F}_f^{n,m}$ one has that
$$\inf f=-\inf_{x^*\in C}f^g(x^*),$$ which means
$$\inf f\leq \inf f^{gg}\leq \inf f.$$ Therefore $\inf f=\inf
f^{gg}$.
\item[ii)] $f^{gg}(x_0)\leq f(x_0)=\inf f=\inf f^{gg}\leq
f^{gg}(x_0),$ then $f^{gg}(x_0)=\inf f^{gg}$.
\end{enumerate}
\end{proof}

It would be interesting to get a duality scheme which generalizes
others like the Lagrangian one. By example, let us consider
$$g(x,x^*)=\left\{\begin{array}{cl}
  f(x)+h^*(x^*) & \mbox{if }x\in dom(f),\ x^*\in dom(h^*) \\
  0 & \mbox{if }x\notin dom(f),\ x^*\in dom(h^*) \\
  +\infty & \mbox{if }x^*\notin dom(h^*)
\end{array} \right.$$ where $h^*$ is the objective function of
problem $(Q)$ (section 2) and if in addition $h$ is l.s.c. at 0,
one has that $g\in \mathcal{F}_f^{n,p}$ (with $C=dom(h^*)$), even
more $f^g\equiv h^*$:
$$f^g(x^*)=\sup_{x\in \R^n} \{g(x,x^*)-f(x)\},\ x^*\in C,$$
but $$g(x,x^*)-f(x)= \left\{\begin{array}{cl}
  f(x)+h^*(x^*)-f(x) & if\ x\in dom(f) \\
  -\infty & if\ x\notin dom(f)
\end{array} \right. \Longrightarrow$$
$$g(x,x^*)-f(x)=\left\{\begin{array}{cl}
  h^*(x^*) & if\ x\in dom(f) \\
  -\infty & if\ x\notin dom(f),
\end{array} \right.$$ thus $f^g\equiv h^*$ and the classic duality
is recovered. (Later on, we will exhibit a more interesting $g\in
\mathcal{F}_f^{n,p}$, at least in the case of a restricted
problem.)\vspace{2ex}

\noindent {\bf Examples:} In the following examples, $f:\R
\rightarrow \R \cup \{+\infty\}$ and $g\in \mathcal{F}^{1,1}$.

\begin{enumerate}
\item[1)] $f(x)=x^2$ and $g(x,x^*)=(xx^*)^2$ ($C=\R$).
Calculate $f^g$:
$$f^g(x^*)=\sup_{x\in
\R}\{(xx^*)^2-x^2\}=\left\{\begin{array}{cc}
  0 & if\ |x^*|\leq 1 \\
  +\infty & if\ |x^*|> 1.
\end{array} \right.$$
Then $$g'(x,x^*)=f(x)+f^g(x^*)=\left\{\begin{array}{cc}
  x^2 & if\ |x^*|\leq 1 \\
  +\infty & if\ |x^*| >1
\end{array}\right.$$ and $g'\in \mathcal{F}_f^{1,1}$.
\item[2)] Take now $$f(x)=\left\{\begin{array}{cc}
  x^2 & if\ x\geq 0 \\
  +\infty & if\ x<0
\end{array}\right. $$ and $$g(x,x^*)=\left\{\begin{array}{cc}
   \frac{1}{xx^*+1} & if\ x,x^*\geq 0 \\
   0 & if\ x<0,\ x^*\geq 0\\
  +\infty & \mbox{otherwise.}
\end{array} \right.$$ $C=[0,+\infty[$,
$$f^g(x^*)=\sup_{x\in \R}\{g(x,x^*)-f(x)\}=
  \displaystyle \sup_{x\geq 0}\left\{\frac{1}{xx^*+1}-x^2\right\}, \ x^*\in C$$

$$\Longrightarrow f^g(x^*)= 1, \ x^*\in C.$$
Then $$g'(x,x^*)=f(x)+f^g(x^*)=\left\{\begin{array}{cc}
  x^2+1 & if\ x,x^*\geq 0 \\
  +\infty & \mbox{otherwise}
\end{array}\right.$$ and $g\notin \mathcal{F}_f^{1,1}$.
\item[3)] Let $f(x)=\exp(x)$ and $g(x,x^*)=\exp(x+x^*)$ ($C=\R$).
Then
$$f^g(x^*)=\sup_{x\in \R}
\{\exp(x)(\exp(x^*)-1)\}=\left\{\begin{array}{cc}
  0 & if\ x^*\leq 0 \\
  +\infty & x^*>0.
\end{array}\right.$$ And $$g'(x,x^*)=f(x)+f^g(x^*)=\left\{
\begin{array}{cc}
  \exp(x) & if\ x^*\leq 0 \\
  +\infty & if\ x^*> 0.
\end{array}\right.$$ Thus $g\in \mathcal{F}_f^{1,1}._{\Box}$
\end{enumerate}

Now, take two particular $g\in \mathcal{F}^{n,n}$ which have some
good properties (see \cite{RUB}).

Firstly, consider $f\in\mathcal{F}^{n}$ such that $dom(f)\subset
\R^n_+$ then take $g\in \mathcal{F}^{n,n}$ defined as follows:
$$g(x,x^*):=\left\{\begin{array}{cc}
  \langle x^*,x\rangle^+ & x\in \R^n_+,\ x^*\in \R^n_+\\
  0 & x\notin \R^n_+,\ x^*\in \R^n_+ \\
  +\infty & x^*\notin \R^n_+,
\end{array}\right.$$ where $\displaystyle\langle x^*,x\rangle^+=
\max_{i=1,\dots,n}\{x_i^*x_i\}$ (here $C=\R^n_+$). It is easy to
prove that this function belongs to $\mathcal{F}^{n,n}_f$ and also
satisfies the following:\begin{enumerate}
\item[$(*)$] $g(\cdot,x^*):\R^n_+\rightarrow \R$ and
$g(x,\cdot):\R^n_+\rightarrow \R$ are convex, l.s.c. and IPH
(increasing positively homogeneous) functions (with $x_1\leq x_2$
if and only if $x_{1i}\leq x_{2i}$ for all $i=1,\ldots,n$) for
every $x^*,x \in \R^n_+$.
\end{enumerate}
Consider now the set of functions $H_g$ defined as:
$$H_g:=\{h:\R^n\rightarrow \R\mbox{ such that }\exists
(x^*,\lambda)\in \R^n_+\times \R \mbox{ with
}h(x)=g(x,x^*)-\lambda\}.$$ With this notation, define the support
set of $f$ with respect to $H_g$, $supp(f,H_g):$
$$supp(f,H_g):=\{h\in H_g:\ h(x)\leq f(x)\ \forall x\in
\R^n\}.$$ It is already proven (see \cite{RUB}) that, since every
$h\in H_g$ is linked to one and only one $(x^*,\lambda)\in
\R^n_+\times \R$, we have
$$supp(f,H_g)=epi(f^g).$$ Let us see what properties are satisfied by $f^{gg}$ and
$supp(f,H_g)$.
\begin{enumerate}
\item[$\bullet$] $\displaystyle f^{gg}(x)=\sup_{x^*\in \R^n_+}
[g(x,x^*)-f^g(x^*)] \Longrightarrow$
$$f^{gg}(x)=\left\{\begin{array}{cc}
  \displaystyle\sup_{x^*\in \R^n_+}[\langle x^*,x\rangle^+
  -f^g(x^*)] & x\in \R^n_+, \\
  \displaystyle\inf_{y\in \R^n} f(y) & x\notin \R^n_+,
\end{array}\right.$$ $\left(\mbox{recall that
}\displaystyle\inf_{y\in \R^n} f(y) =\sup_{x^*\in
\R^n_+}[-f^g(x^*)]\right)$. Since the functions $\langle
x^*,\cdot\rangle^+-f^g(x^*)$ are increasing convex and l.s.c. for
every $x^*\in \R^n_+$, then $f^{gg}$ is increasing convex and
l.s.c. on $\R^n_+$.

\item[$\bullet$]$\displaystyle
f^g(x^*)=\sup_{x\in \R^n}[g(x,x^*)-f(x)]=\sup_{x\in dom(f)}
[\langle x^*,x\rangle^+ -f(x)]$ as before, the functions $\langle
\cdot,x\rangle^+ -f(x)$ are convex increasing and l.s.c for every
$x\in dom(f)\subset \R^n_+$, therefore $f^g$ is a convex
increasing and l.s.c. function. With this, one has that
$supp(f,H_g)=epi(f^g)$ is closed, convex and if $f$ is
non-negative then for every $(x^*,\lambda)\in supp(f,H_g)$,
$t(x^*,\lambda)\in supp(f,H_g)$ for all $t\in [0,1]$.
\end{enumerate}

Thanks to \cite{RUB} another interesting property can also be
shown. Since $\R^n_+$ is a convex closed cone, one has that for
all $x\in \R^n_+$,
$$\partial f^{gg}_{x_i}(1)\neq \emptyset,\ \forall i=1,\ldots,n$$
where $x_i$ is a vector with its components equal to zero except
the i-th component which is equal to the i-th component of $x$ and
$f^{gg}_{x_i}:[0,+\infty]\rightarrow \R_{+\infty}$ is defined by
$$f^{gg}_{x_i}(t):=f^{gg}(tx_i),\ \forall t\in [0,+\infty].$$
Secondly, if we consider
$$g(x,x^*):=\left\{\begin{array}{cc}
  \langle x^*,x\rangle^- & x\in \R^n_+,\ x^*\in \R^n_+\\
  0 & x\notin \R^n_+,\ x^*\in \R^n_+ \\
  +\infty & x^*\notin \R^n_+,
\end{array}\right.$$ where $\displaystyle\langle x^*,x\rangle^-=
\min_{i\in I_+(x^*)}\{x_i^*x_i\}$, $I_+(x^*)=\{i:l_i>0\}$ (here
also $C=\R^n_+$), then $f^{gg}$ is a l.s.c. increasing
convex-along-rays (ICAR) function (\cite{RUB}) and we can find in
chapter 3, section 3 of \cite{RUB} several results about this kind
of function including a condition which guarantees that in some
points, the general sub-differential of the function $f^{gg}$ is
nonempty. \vspace{2ex}

\noindent{\bf Remark: } The previous results are valid for every
$g \in \mathcal{F}^{n,n}_f$ which satisfies $(*)$.

\subsection{Particular Case: Classical Lagrangian Duality}
Let $$(P):\ \min_{x\in A}f(x)$$ be a typical minimization problem,
where
$$A:=\{x\in \R^n:\ h_i(x)\leq 0,\ \forall i=1,\ldots,m\},$$
$$f:A\rightarrow \R,\ h_i:\R^n\rightarrow \R,\mbox{
are convex l.s.c functions with }i=1,\ldots,m.$$ Remember that
(see \cite{J.P.SOSA.OC}) the following is the well known dual
problem:
$$(D_L):\ \min_{\lambda^*\geq 0}\sup_{x\in A}\{\langle
\lambda^*,-h(x)\rangle -f(x)\},$$ $h(x)=(h_1(x),\ldots, h_m(x))$.
Moreover, $x_0$ is a solution of $(P)$ and $\lambda^*_0$ is a
solution of $(D_L)$ if and only if $(x_0,\lambda^*_0)$ is a saddle
point of the Lagrangian function $L$, given by
$$L(x,\lambda^*):=f(x)+\langle \lambda^*,h(x)\rangle,$$ which
means,
$$L(x_0,\lambda^*)\leq L(x_0,\lambda_0)\leq L(x,\lambda^*_0),\
\forall x\in A,\ \forall \lambda^*\geq 0.$$ Define now $g_1:\R^m
\times \R^m \rightarrow \overline{\R}$, as follows:
$$g_1(y,\lambda^*)=\left\{\begin{array}{cl}
  \langle \lambda^*,y\rangle & if\ y\geq 0,\ \lambda^*\geq 0 \\
  0 & if\ y\ngeq 0,\ \lambda^*\geq 0\\
  +\infty & if\ \lambda^*\ngeqslant 0.
\end{array} \right.$$
It can be seen that $g_1\in \mathcal{F}^{m,m}$, with $C=\R^m_+$.
Let
$$g(x,\lambda^*)=g_1(-h(x),\lambda^*),$$ then $g\in
\mathcal{F}^{n,m}$. Calculating $\overline{f}^g(\lambda^*)$ with
$\lambda^*\in C$ and
$$\overline{f}(x)=\left\{\begin{array}{cl}
  f(x) & if\ x\in A \\
  +\infty & if\ x\notin A.
\end{array}\right.$$
$$\mbox{Then }\overline{f}^g(\lambda^*)=\sup_{x}
\{g(x,\lambda^*)-\overline{f}(x)\}=\sup_{x\in A} \{ \langle
\lambda^* ,-h(x)\rangle -f(x)\},$$ which means
$$\overline{f}^g(\lambda^*)= \sup_{x\in A}\{\langle \lambda^*
,-h(x)\rangle -f(x)\},\mbox{ for every } \lambda^*\in C=\R^m_+$$
and thus the classical Lagrangian duality is recovered.

\section{The Equilibrium Problem}
Let $f:K\times \R^n \rightarrow \overline{\R}$, where $K\subset
\R^n$ is a non-empty closed convex set, be such that
\begin{enumerate}
\item[i)] $f(x,\cdot):\R^n\rightarrow \overline{\R}$ is
a convex l.s.c. function for every $x\in K$.
\item[ii)] $f(x,x)=0,\ \forall x\in K$.
\end{enumerate}
The Equilibrium Problem is defined as follows:
$$(EP):\mbox{ Find }x\in K \mbox{ such that }f(x,y)\geq 0,
\forall y\in K.$$ Pseudo-monotone functions, which are defined as
follows, are consider in \cite{SOSA}. A function $\varphi:\R^n
\times \R^n \rightarrow \overline{\R}$ would be called
pseudo-monotone if for every $x,y\in \R^n$ with $\varphi(x,y)\geq
0$, we have $\varphi(y,x)\leq 0$.

\begin{proposition} Let $g\in \mathcal{F}^{n,n}$ and $C\subset
\R^n$ be the non-empty closed set linked to $g$. If $g$ is
pseudo-monotone, then $g(x,y)=0,\ \forall (x,y)\in C\times C$.
\end{proposition}
\begin{proof} Since $g\in \mathcal{F}^{n,n}$, then $\forall
(x,y)\in C\times C,\ g(x,y)\geq 0$ and $g(y,x)\geq 0$. Therefore
if $g$ is pseudo-monotone, it can be said that $g(y,x)\leq 0$ and
$g(y,x)=0,\ \forall (x,y)\in C\times C$.
\end{proof}

Observe that this proposition affirms that the only
pseudo-monotone G-coupling function is the null function.

Take now $m\in \N$ and $g\in \mathcal{F}^{n,m}$. Consider now for
every $x\in K$:
$$(P^x):\ \inf_{y\in \R^n} \overline{f}(x,y) \qquad \qquad
(D^x_g):\ \inf_{x^*\in C }\overline{f}_x^g(x^*)$$ where
$\overline{f}^g_x(x^*)=\displaystyle \sup_{y\in \R^n}
\{g(y,x^*)-\overline{f}(x,y)\}$ for every $x^*\in C$ ($C$ is the
non-empty closed set linked to $g$) and
$$\overline{f}(x,y) =\left\{\begin{array}{cc}
  f(x,y) & y\in K \\
  +\infty & y\notin K. \end{array} \right.$$

It would be interesting to know if there exist $m\in \N$ and a
$g\in \mathcal{F}^{n,m}$ which satisfy the \emph{Zero Duality Gap
Property} (ZDGP):
$$\inf_{y\in \R^n} \overline{f}(x,y)+ \inf_{x^*\in
C}\overline{f}^g_x(x^*)=0,\ \forall x\in F$$ where
$$F:=\left\{x\in K:\ \displaystyle\inf_{y\in K} f(x,y)\neq
-\infty\right\}.$$ (Notice that if $F$ is empty, $(EP)$ would have
no solutions.) If there exists such a $g$, the following are
equivalent: \begin{enumerate}
\item[$\bullet$] $\overline{x}$ is a solution of $(EP)$.
\item[$\bullet$] $\displaystyle \inf_{y\in \R^n}
\overline{f}(\overline{x},y)=0$.
\item[$\bullet$] There exists $\overline{x}\in K$ such that
$\displaystyle \inf_{x^*\in C}
\overline{f}^g_{\overline{x}}(x^*)=0.$
\end{enumerate}

\begin{lemma}\label{lemmexistZDGP}
There exists $g\in \mathcal{F}^{n,n}$ which satisfies the ZDGP.
\end{lemma}
\begin{proof} Define $g\in \mathcal{F}^{n,n}$ by:
$$g(y,x^*)=\left\{\begin{array}{cc}
  \langle x^*,y\rangle & if\ x^*\in K^+,\ y\in K \\
  0 & if\ x^*\in K^+,\ y\notin K\\
  +\infty & x^*\notin K^+,
\end{array} \right.$$
where $K^+:=\{x^*\in \R^n:\langle x^*,x\rangle\geq 0,\ \forall
x\in K\}$, in this case $C=K^+$. Calculate
$\overline{f}_x^g(x^*)$, with $x\in F$, $x^*\in C$:
$$\overline{f}_x^g(x^*)=\sup_{y\in K}\{\langle x^*,y\rangle
-f(x,y)\}.$$ It is clear that $$\inf_{x^*\in
C}\overline{f}_x^g(x^*) \leq \overline{f}_x^g(0)=-\inf_{y\in \R^n}
\overline{f}(x,y)=-\inf_{y\in K }f(x,y),$$ but for every $x^*\in
C$, $y\in \R^n$ one has that
$\overline{f}_x^g(x^*)+\overline{f}(x,y)\geq 0$, which implies
that we always have $\displaystyle \inf_{x^*\in
C}\overline{f}_x^g(x^*) \geq -\inf_{y\in \R^n} \overline{f}(x,y)$.
Therefore
$$\inf_{x^*\in C}\overline{f}_x^g(x^*) =-\inf_{y\in \R^n}
\overline{f}(x,y) \Longrightarrow \inf_{x^*\in
C}\overline{f}_x^g(x^*)+ \inf_{y\in \R^n} \overline{f}(x,y)=0.$$
Finally, there exists $g\in \mathcal{F}^{n,n}$ which satisfies the
ZDGP.
\end{proof}

Let us give now another function $g\in \mathcal{F}^{n,n}$ which
also satisfies the ZDGP. In this case, this $g$ will generate a
duality scheme which has been already studied in \cite{JEML.SOSA}.

Let $i_K:\R^n\rightarrow \overline{\R}$ be defined by
$\displaystyle i_K(x^*):=\inf_{x\in K}\langle x^*,x\rangle$ and
$K^*:=\{x^*\in \R^n:\ i_K(x^*)>-\infty\}$, the effective domain of
$i_K$. (Since $i_K$ is a concave u.s.c function, then $K^*$ is a
closed convex set.) Define then:
$$g(y,x^*)=\left\{\begin{array}{cl}
  \langle x^*,y\rangle -i_K(x^*), & \mbox{if }x^*\in K^*,\ y\in
  K\\
  0 & \mbox{if }x^*\in K^*,\ y\notin K \\
  +\infty & \mbox{if }x^*\notin K^*,
\end{array}\right.$$ in this case $C=K^*$. Calculate now
$\overline{f}^g_x(x^*)$ for $x\in F$:
$$\begin{array}{ccl}
  \overline{f}^g_x(x^*) & = & \displaystyle \sup_{y\in \R^n}
  [g(y,x^*)- \overline{f}(x,y)]\\
   & = & \displaystyle \sup_{y\in K}[\langle x^*,y\rangle
   -i_K(x^*)-f(x,y)] \\
   & = & \displaystyle \sup_{y\in K}[\langle x^*,y\rangle
   -f(x,y)]-i_K(x^*),
\end{array}$$ but $\displaystyle \overline{f}^*_x(x^*)=\sup_{y\in
\R^n} [\langle x^*,y\rangle-\overline{f}(x,y)]=\sup_{y\in K}
[\langle x^*,y\rangle-f(x,y)]$, then
$$\overline{f}_x^g(x^*)=\overline{f}_x^*(x^*)-i_K(x^*).$$

\begin{proposition}
The function $g$ defined above, satisfies the ZDGP.
\end{proposition}
\begin{proof} Similar to Lemma \ref{lemmexistZDGP}.
\end{proof}

In \cite{JEML.SOSA} a very interesting result can be found.

\begin{theorem}\label{jemlws}
$\overline{x}$ is a solution of $(EP)$ if and only if, there
exists $x^*\in K^*$ such that
$\overline{f}_{\overline{x}}^*(x^*)-i_K(x^*)=0$.
\end{theorem}

This result does not only say that $\displaystyle \inf_{x^*\in
K^*}\overline{f}^g_{\overline{x}}(x^*)=0$, but that the dual
problem $(D_g^{\overline{x}})$ has a solution. In order to prove
this, we need first the following lemma.

\begin{lemma} For every $x^*\in K^*$ and $x\in F$, one has:
$$\overline{f}_x^*(x^*)-i_K(x^*)\geq 0.$$
\end{lemma}
\begin{proof} From Fenchel's inequality, we have for every $x\in F$
fixed
$$\langle x^*,y\rangle-\overline{f}^*_x(x^*)\leq
\overline{f}(x,y),\ \forall x^*\in K^*,\ \forall y\in \R^n.$$
Then, taking $\displaystyle \inf_{y\in K}$:
$$i_K(x^*)-\overline{f}^*_x(x^*)\leq\inf_{y\in K}f(x,y)\leq 0,$$
where the last inequality occurs, since $\displaystyle \inf_{y\in
K}f(x,y)\leq f(x,x)=0$. This means,
$\overline{f}_x^*(x^*)-i_K(x^*)\geq 0$.
\end{proof}

\begin{proof}
{\bf of Theorem \ref{jemlws}:} \begin{enumerate}
\item[$\bullet$] If $\overline{x}$ is a solution of $(EP)$,
then $$\overline{f}_{\overline{x}}(\overline{x})=
f_{\overline{x}}(\overline{x})= 0= \min_{y\in
K}\overline{f}(\overline{x},y)=\min_{y\in K}f(\overline{x},y).$$
Then, there exists (see \cite{ROCK}) $x^*\in \partial
f_{\overline{x}}(\overline{x})\cap (-N_K(\overline{x}))$ (where
$N_K(\overline{x})$ stands for the normal cone of $K$ at
$\overline{x}$) and thus
$$\overline{f}^*_{\overline{x}}(x^*)\leq
f^*_{\overline{x}}(x^*)=f_{\overline{x}}(\overline{x})+
f^*_{\overline{x}}(x^*)=\langle x^*,\overline{x}\rangle
=i_K(x^*),$$ which means
$\overline{f}^*_{\overline{x}}(x^*)-i_K(x^*)\leq 0$. But thanks to
the previous lemma, this implies that
$\overline{f}^*_{\overline{x}}(x^*)-i_K(x^*)=0$.
\item[$\bullet$] Take $y\in K$ and suppose that there
exists $x^*\in K^*$ such that
$\overline{f}^*_{\overline{x}}(x^*)=i_K(x^*)$, then
$$f(\overline{x},y)=\overline{f}(\overline{x},y)\geq \langle
x^*,y\rangle -\overline{f}^*_{\overline{x}}(x^*)= \langle
x^*,y\rangle -i_K(x^*)\geq 0.$$ And thus, $\overline{x}$ is a
solution of $(EP)$.
\end{enumerate}
\end{proof}

\subsection{Particular Case: The Complementarity Problem}
A particular case of the Equilibrium Problem is the
Complementarity Problem, which is defined as follows:
$$(CP):\mbox{ Find }x\in K\mbox{ such that }T(x)\in K^+\mbox{ and
}\langle T(x),x \rangle=0,$$ where $K\subset \R^n$ is a closed
convex cone and $ T:K\rightarrow \R^n$ is a continuous function.
Considering $f(x,y)=\langle T(x),y-x \rangle$ with $x$ and $y$ in
$K$, the solution set of the $(CP)$ is equal to the solution set
of $(EP)$ related to $f$.

Let us take $g\in \mathcal{F}^{n,n}$ defined as in the beginning
of this section:
$$g(y,x^*)=\left\{\begin{array}{cl}
  \langle x^*,y\rangle & if\ x^*\in K^+,\ y\in K \\
  0 & if\ x^*\in K^+,\ y\notin K\\
  +\infty & x^*\notin K^+.
\end{array} \right.$$ Calculate $\overline{f}^g_x(x^*)$
($x\in K,\ x^*\in C=K^+$):
$$
\overline{f}^g_x(x^*)=\sup_{y\in
\R^n}\{g(y,x^*)-\overline{f}(x,y)\}= \sup_{y\in
K}\{g(y,x^*)-f(x,y)\}$$ $$\Rightarrow
\overline{f}^g_x(x^*)=\sup_{y\in K}\{\langle
  x^*-T(x),y\rangle\}+\langle T(x),x \rangle$$
$$\Rightarrow
\overline{f}^g_x(x^*)=\left\{\begin{array}{cc}
  \langle T(x),x\rangle  & x^*-T(x)\in
  K^- \\
  +\infty & \mbox{otherwise.}
\end{array}\right.$$
But $x^*-T(x)\in K^-$ is equivalent to the statement that $$x\in
T^{-1}(x^*+K^+)\subset T^{-1}(K^+),$$ and this inclusion is true
since $K$ is a closed convex cone and $x^*\in K^+$. Then

$$\overline{f}^g_x(x^*)=\left\{\begin{array}{cc}
  \langle T(x),x\rangle  & x\in
T^{-1}(x^*+K^+) \\
  +\infty & \mbox{otherwise.}
\end{array}\right.$$
Therefore,
$$\overline{f}(x,y)+\overline{f}^g_x(x^*)=\left\{
\begin{array}{cc}
  \langle T(x),y\rangle & if\ x\in T^{-1}(x^*+K^+),\
  y\in K \\
  +\infty & \mbox{otherwise.}
\end{array}\right.$$

Calculate now the set $F$:
\begin{enumerate}
\item[$\bullet$] If $x\in K$ is such that $T(x)\in K^+$ then
$$\inf_{y\in K}f(x,y)=\inf_{y\in K}\langle
T(x),y-x\rangle=-\langle T(x),x\rangle\neq -\infty,$$ which means
$x\in F$.
\item[$\bullet$] If $x\in K$ is such that $T(x)\notin K^+$ then
there exists $y_0\in K$ satisfying $\langle T(x),y_0\rangle<0$.
Thus $$\lim_{n\rightarrow +\infty}f(x,ny_0)=-\infty=\inf_{y\in
K}f(x,y),$$ which means $x\notin F$.
\end{enumerate} All these imply that $F=T^{-1}(K^+)$.\\
\\ Then, given $x_0\in F$, there exists $x^*\in K^+$ (for example
$x^*=0$) such that\\ $x_0\in T^{-1}(x^*+K^+)$ and:
$$\inf_{y\in \R^n}\overline{f}(x_0,y)+\inf_{x^*\in
C}\overline{f}^g_{x_0}(x^*)=0,$$ but $x_0\in F$ was chosen
arbitrarily, therefore
$$\inf_{y\in \R^n}
\overline{f}(x,y)+\inf_{x^*\in C}\overline{f}^g_{x}(x^*)=0,\
\forall x\in F.$$ Finally, there exists $\overline{x}\in K$ such
that
\begin{equation}\label{nec&suf cond for (CP)}\inf_{x^*\in
C}\overline{f}^g_{\overline{x}}(x^*)=0 \Longleftrightarrow
T(\overline{x})\in K^+\mbox{ and }\langle
T(\overline{x}),\overline{x}\rangle =0. \end{equation}

In \cite{COTTLE} the $(CP)$ is considered, when $K=K^+=\R^n_+$ and
$T$ is an affine operator, in other words, the case of the Linear
Complementarity Problem $(LCP)$. For studying this problem, they
propose the following:\\ \\
$\overline{x}$ is a solution of $(LCP)$, if and only if
$\overline{x}$ satisfies:
$$ \overline{x}\in \R^n_+,\ T(\overline{x})\in \R^n_+,\mbox{ and }\langle
T(\overline{x}),\overline{x}\rangle =0.$$ It is immediate to see
that this proposal is identical to (\ref{nec&suf cond for (CP)}),
therefore by using this $g\in \mathcal{F}^{n,n}$ (the one used at
the beginning of this section) we generate a dual problem of
$(LCP)$ which has been treated in other works.

\section*{Conclusions}
This work gives a basis for a new theory, which we called
G-coupling functions. Logically there are many things to explore,
by example:
\begin{enumerate}
\item[$\bullet$] Using G-coupling functions for a Perturbation
Theory.
\item[$\bullet$] Using G-coupling functions for generating primal,
dual and primal-dual algorithms.
\item[$\bullet$] Analyze, using G-coupling functions, the
Variational Inequality Problem for the case of non-monotone
operators.
\item[$\bullet$] Given $f$ look for a G-coupling function such that $f$ is abstract
convex with respect to the class of elementary function induced
G-coupling function (see \cite{RUB}).
\end{enumerate}

{\bf Acknowledgement}: We would like to express our deep thanks to
Dr. David Yost for all his support and suggestions. We also thank
to an anonymous referee for his valuable help to correct the
paper.


\begin{thebibliography}{99}\rm

\bibitem{AUS-TEB}
Auslander, Alfred; Teboulle, Marc.
\newblock {\em Asymptotic Cones and Functions in Optimization and
Variational Inequalities.}
\newblock Springer-Verlag New York, INC. (2003).


\bibitem{AVRIEL}
Avriel, Mordecai
\newblock {\em Nonlinear Programming.}
\newblock Prentice-Hall, INC., Englewood Cliffs, New Jersey (1976).



\bibitem{Blum-Oettli}
Blum, E.; Oettli, W.
\newblock {\em From optimization and variational inequalities to equilibrium problems.}
\newblock Math. Student 63, 123-145 (1994).



\bibitem{BUR}
Burachik, Regina.
\newblock {\em Generalized Proximal Point Methods for the
Variational Inequalty Problem.}
\newblock Tese de Doutorado, IMPA. (1995).



\bibitem{COTTLE}
Cottle, R.W.; Pang, J.S.; Venkateswaran, V.
\newblock {\em Sufficient matrices and the linear complementarity
problem.}
\newblock Linear Algebra and its Applications 114/115, pp. 231-249 (1989).


\bibitem{J.P.SOSA.OC}
Crouzeix, Jean Pierre; Oca\~{n}a, Eladio; Sosa, Wilfredo.
\newblock {\em An\'{a}lisis Convexo.}
\newblock Monogaf\'{\i}a del IMCA (2003).


\bibitem{JEML.SOSA}
Mart\'{i}nez-Legaz, Juan Enrique and Sosa, Wilfredo.
\newblock {\em Duality for Equilibrium Problems.}
\newblock Journal of Global Optimization 35, pp. 311-319 (2006).


\bibitem{MOS}
Maugeri, A.; Oettli, W.; Schl\"{a}ger, D.
\newblock {\em A flexible form of Wardrop's principle for traffic equilibria with side constraints.}
\newblock Rendiconti del Circolo Matematico di Palermo, Serie II, Supp. 48, 185-193 (1997).


\bibitem{OS}
Oettli, W.; Schl\"{a}ger, D.
\newblock {\em Conjugate functions for convex and nonconvex duality.}
\newblock Journal of Global Optimization 13, pp. 337-347 (1998).


\bibitem{ROCK}
Rockafellar, R.T.
\newblock {\em Convex Analysis.}
\newblock Princeton University Press, Princeton, New Jersey (1970)


\bibitem{RUB}
Rubinov, Alexander.
\newblock {\em Abstract Convexity and Global Optimization.}
\newblock Kluwer Academic Publishers, Dordrecht/Boston/London (2000).


\bibitem{RUB-YANG}
Rubinov, Alexander and Yang Xiaoqi.
\newblock {\em Lagrange-type Functions in Constrained non-convex
Optimization.}
\newblock Kluwer Academic Publishers, Dordrecht/Boston/London (2002).


\bibitem{SOSA}
Sosa Sandoval, Wilfredo.
\newblock {\em Iterative Algorithms for the abstract Equilibrium
Problem.}
\newblock Tese de Doutorado, IMPA. (1999).


\bibitem{SOSA2}
Sosa Sandoval, Wilfredo.
\newblock {\em Introducci\'{o}n a la Optimizaci\'{o}n,
Programaci\'{o}n Lineal.}
\newblock XVIII Coloquio de la Sociedad Matem\'{a}tica Peruana
(2000).


\bibitem{WRIGHT}
Wright, Stephen J.
\newblock {\em Primal-Dual Interior Point Methods.}
\newblock Society for Industrial and Applied Mathematics (S.I.A.M.),
Philadelphia, (1997).


\bibitem{YANG}
Yang, X. Q.
\newblock {\em On the GAP functions of Prevariational Inequalties.}
\newblock Journal of Optimization Theory and Applications Vol. 116,
No 2, pp. 437-452 (2003).


\bibitem{ZANG}
Zangwill, Willard
\newblock {\em Nonlinear Programming. A unified Approach}
\newblock Prentice-Hall, Inc., Englewood Cliffs, New Jersey
(1969).
\end{thebibliography}
\end{document}